\documentclass[11pt,leqno]{amsart}
\numberwithin{equation}{section}
\allowdisplaybreaks[1]
\usepackage{amsmath,graphicx,amssymb,amscd}
\usepackage[initials]{amsrefs}
\newtheorem{lem}{Lemma}[section]
\newtheorem{thm}{Theorem}[section]

\newtheorem{corollary}{Corollary}[section]
\newtheorem{proposition}{Proposition}[section]
\newtheorem{rem}{Remark}[section]

\DeclareMathOperator{\arcsinh}{arcsinh}

\title{Lyapunov-type inequalities for a Sturm-Liouville problem of the
 one-dimensional $p$-Laplacian}
\author[S. Takeuchi]{Shingo Takeuchi${}^\dag$}
\author[K. Watanabe]{Kohtaro Watanabe${}^\ddag$}
\email[Shingo Takeuchi]{shingo@shibaura-it.ac.jp}
\email[Kohtaro Watanabe]{wata@nda.ac.jp}
\address[Shingo Takeuchi]{Department of Mathematical Sciences, 
Shibaura Institute of Technology,
307 Fukasaku, Minuma-ku, Saitama-shi, Saitama 337-8570, Japan}
\address[Kohtaro Watanabe]{
Department of Computer Science, National Defense Academy, 
1-10-20 Hashirimizu, Yokosuka 239-8686, Japan}
\thanks{$\dag$ This work is partially supported by 
the Grant-in-Aid for Scientific Research (C) (No. 17K05336)
from Japan Society for the Promotion of Science.\\
\indent
$\ddag$ This work is partially supported by 
the Grant-in-Aid for Scientific Research (C) (No. 18K03387)
from Japan Society for the Promotion of Science.}
\subjclass[2010]{}
\begin{document}
\begin{abstract}
This article considers the eigenvalue problem for the Sturm-Liouville problem including $p$-Laplacian
\begin{align*}
\begin{cases}
\left(\vert u'\vert^{p-2}u'\right)'+\left(\lambda+r(x)\right)\vert u\vert ^{p-2}u=0,\,\, x\in (0,\pi_{p}),\\
u(0)=u(\pi_{p})=0,
\end{cases}
\end{align*} 
where $1<p<\infty$, $\pi_{p}$ is the generalized $\pi$ given by $\pi_{p}=2\pi/\left(p\sin(\pi/p)\right)$, $r\in C[0,\pi_{p}]$ and $\lambda<p-1$.
Sharp Lyapunov-type inequalities, which are necessary conditions for the existence of nontrivial solutions of the above problem are presented.
Results are obtained through the analysis of variational problem related to a sharp Sobolev embedding and generalized trigonometric and hyperbolic functions. 
\end{abstract}
\maketitle
\section{Introduction}
Let $1<p<\infty$. This article considers the eigenvalue problem for the Sturm-Liouville problem including $p$-Laplacian:
\begin{align}\label{SLP}
\begin{cases}
\left(\vert u'\vert^{p-2}u'\right)'+\left(\lambda+r(x)\right)\vert u\vert ^{p-2}u=0,\quad x\in (0,\pi_{p}),\\
u(0)=u(\pi_{p})=0,
\end{cases}
\end{align} 
where $\pi_{p}$ is the generalized $\pi$ given by 
\begin{align}\label{PIp}
\pi_{p}=2\int_{0}^{1}\frac{dt}{(1-t^p)^{1/p}}=\frac{2\pi}{p\sin{(\pi/p)}},
\end{align}
$r\in C[0,\pi_{p}]$ and $\lambda<p-1$. 
We present sharp Lyapunov-type inequalities for \eqref{SLP}, which are necessary conditions for the existence of nontrivial solutions of \eqref{SLP}.
Here, a function $u$ is called a solution of \eqref{SLP} if 
$u \in W^{1,p}_0(0,\pi_p)$ and $u$ satisfies the equation in \eqref{SLP}
in the distribution sense. It is easily seen that the solution $u$ of \eqref{SLP}
has the smoothness properties $u,\ |u'|^{p-2}u' \in C^1[0,\pi_p]$
and therefore satisfies \eqref{SLP} in the classical sense.

In the case $p=2$, the necessary conditions for the existence of 
nontrivial solutions of \eqref{SLP} are known as
\begin{align}\label{lyapunov2}
\|r_{+}\|_{L^{1}(0,\pi)}>
\begin{cases}
2\sqrt{\lambda}\cot\left(\frac{\sqrt{\lambda}\pi}{2}\right),&0<\lambda<1,\\
\frac{4}{\pi},&\lambda=0,\\
2\sqrt{-\lambda}\coth \left(\frac{\sqrt{-\lambda}\pi}{2}\right),&\lambda<0,
\end{cases}
\end{align}
where $r_{+}(x)=\max (r(x),0 )$. 
Especially, the inequality \eqref{lyapunov2} for $\lambda=0$ is originally called the \textit{Lyapunov inequality} 
and one can see several proofs of the inequality in Chapter 1 of Pinasco \cite{Pinasco2013} (see also Ca\~{n}ada and Villegas \cite{Canada-Villegas2015}
for the $L^q$-norm of $r_+$, $1 \leq q \leq \infty$).  
The inequalities for $\lambda \in (0,1)$ and $\lambda \in (-\infty,0)$ were obtained by Borg \cite{Borg1944}
(see also Ha \cite{Ha1998} for $\lambda \in (0,\infty)$), 
and rely on the construction of Green's function (the reproducing kernel for $H^{1}_{0}(0,\pi)$) for the problem:
\begin{align*}
\begin{cases}
u''+\lambda u=0,\quad x\in(0,\pi),\\
u(0)=u(\pi)=0.
\end{cases}
\end{align*}

On the other hand, for problem \eqref{SLP} with general $p \in (1,\infty)$ and $p\neq 2$, we cannot take this approach due to the lack of Green's function. 
This might be the reason why sharp Lyapunov-type inequalities (with concrete constants as \eqref{lyapunov2}) currently available for \eqref{SLP} are limited to the case $\lambda=0$ (see Pinasco \cite{Pinasco2004,Pinasco2013} and Watanabe \cite{Watanabe2012}). 

In this paper, we will obtain sharp Lyapunov-type inequalities for \eqref{SLP} 
in the cases of $0<\lambda<p-1$ and $\lambda<0$, respectively. 
These estimates are obtained through the analysis of variational problem related to a sharp Sobolev embedding and generalized trigonometric and hyperbolic functions.

\section{Main Results}
To state the main results, we introduce the definitions and some properties 
of generalized trigonometric and hyperbolic functions 
(see \cite{Kobayashi-Takeuchi,Takeuchi2019,Klen,YHWL2019} and the references given there for more details). 

For $p \in (1,\infty)$, the generalized function $\arcsin_{p}x$ is defined as
\begin{align}\label{arcsinp}
\arcsin_{p}{x}=\int_{0}^{x}\frac{1}{\left(1-t^{p}\right)^{1/p}}\,dt,\quad 0\leq x\leq1,
\end{align}
and hence it follows from \eqref{PIp} that $\arcsin_{p}{1}=\pi_{p}/2$.
The function $\sin_{p}$ is defined as the inverse of $\arcsin_{p}$ on $[0,\pi_{p}/2]$ and its definition domain can be extended to ${\Bbb R}$ 
as a $2\pi_p$-periodic function by means of $\sin_p{x}=\sin_p{(\pi_p-x)}$
and $\sin_p{(-x)}=-\sin_p{x}$.
We see at once that $\sin_p{x}$ is a smooth function in ${\Bbb R}$
and $\sin_{p}x=0$ if and only if $x=n\pi_{p}\ (n\in {\Bbb Z})$.
Thus, the function $\cos_{p}{x}$ is defined as
\begin{align}\label{diff2}
\cos_{p}x=\left(\sin_{p}x\right)', \quad x \in {\Bbb R},
\end{align}
and we put 
\begin{align*}
\cot_{p}x=\frac{\cos_{p}x}{\sin_{p}x},\quad x\in {\Bbb R}\setminus 
\{n\pi_{p}\,|\, n\in {\Bbb Z}\}.
\end{align*}
From \eqref{arcsinp}, the $p$-Pythagorean identity follows:
\begin{align}\label{plusone}
\vert\cos_{p}x\vert^p + \vert\sin_{p}x\vert^{p}=1,\quad x\in {\Bbb R}.
\end{align}
Moreover, from \eqref{plusone}, $\sin_{p}{x}$
is a unique solution to the initial value problem:
\begin{align}\label{ssh}
\begin{cases}
\left(\vert u'\vert^{p-2}u'\right)'+(p-1)\vert u\vert^{p-2}u=0,\\
u(0)=0,\, u'(0)=1.
\end{cases}
\end{align}
Especially, $\lambda_1=p-1$ is the first eigenvalue 
and $\sin_{p}{x}$ is the first eigenfunction of
\begin{align*}
\begin{cases}
\left(\vert u'\vert^{p-2}u'\right)'+\lambda\vert u\vert^{p-2}u=0,\\
u(0)=u(\pi_{p})=0.
\end{cases}
\end{align*}

Similarly, the function $\arcsinh_p{x}$ is defined as
\begin{align}\label{arcsinhp}
\arcsinh_{p}{x}=\int_{0}^{x}\frac{1}{\left(1+t^{p}\right)^{1/p}}\,dt,\quad 0\leq x<\infty.
\end{align}
Since $\arcsinh_{p}{x} \to \infty$ as $x \to \infty$, 
the function $\sinh_{p}$ is defined as the inverse of $\arcsinh_p$ on $[0,\infty)$
and its definition domain can be extended to ${\Bbb R}$ 
as an odd function by means of $\sinh_p{(-x)}=-\sinh_p{x}$.
We see at once that $\sinh_p{x}$ is a monotone increasing 
smooth function in ${\Bbb R}$
and $\sinh_{p}x=0$ if and only if $x=0$.
Thus, the function $\cosh_{p}{x}$ is defined as
\begin{align}\label{diff2h}
\cosh_{p}x=\left(\sinh_{p}x\right)', \quad x \in {\Bbb R},
\end{align}
and we put
\begin{align*}
\coth_{p}x=\frac{\cosh_{p}x}{\sinh_{p}x},\quad x\in {\Bbb R} \setminus \{0\}.
\end{align*}
In this case, from \eqref{arcsinhp}, the $p$-Pythagorean-like identity follows:
\begin{align}\label{minusone}
\left(\cosh_{p}x\right)^p - \vert \sinh_{p}x \vert^{p}=1,\quad x\in {\Bbb R}.
\end{align}
Moreover, from \eqref{minusone}, $\sinh_{p}{x}$
is a unique solution to the initial value problem:
\begin{align}\label{ssh2}
\begin{cases}
\left(\vert u'\vert^{p-2}u'\right)'-(p-1)\vert u\vert^{p-2}u=0,\\
u(0)=0,\, u'(0)=1.
\end{cases}
\end{align}

Now we are in a position to state our main results.
\begin{thm}\label{thm1}
Suppose $1<p<\infty$, $\lambda<\lambda_1$, and \eqref{SLP} has a nontrivial solution.
Then, it holds
\begin{align}\label{LyapunovLP}
\|r_{+}\|_{L^{1}(0,\pi_{p})}>
\begin{cases}
2\left(\frac{\lambda}{\lambda_1}\right)^{\frac{p-1}{p}}\left(\cot_{p}\left(\frac{\pi_p}{2}\left(\frac{\lambda}{\lambda_1}\right)^{1/p}\right)\right)^{p-1},& 0<\lambda<\lambda_1,\\
\frac{2^{p}}{\pi_{p}^{p-1}}, &\lambda=0,\\
2\left(\frac{-\lambda}{\lambda_1}\right)^{\frac{p-1}{p}}\left(\coth_{p}\left(\frac{\pi_p}{2}\left(\frac{-\lambda}{\lambda_1}\right)^{1/p}\right)\right)^{p-1},& \lambda<0.
\end{cases}
\end{align}
Moreover, 
the estimate above is sharp, in the sense that there exists $r\in C[0,\pi_{p}]$ 
for which \eqref{SLP} has a nontrivial solution such that the left-hand side of \eqref{LyapunovLP} can be arbitrarily closed to the right-hand side.  
\end{thm}

The inequality \eqref{LyapunovLP} is a fairly straightforward generalization of \eqref{lyapunov2}.
Here, the estimate for $\lambda=0$ is originally due to Elbert \cite{Elbert1981}
(see also Pinasco \cite{Pinasco2004,Pinasco2013}),
but our proof will be an alternative to those of \cite{Elbert1981,Pinasco2004,Pinasco2013}.

For $\lambda<\lambda_1$, we denote by $C(p,\lambda)$ the right-hand side of \eqref{LyapunovLP}. 
The next result shows that $C(p,\lambda)$ is the best constant of some Sobolev-type inequality. 
\begin{corollary}\label{cor1}
Suppose $1<p<\infty$ and $\lambda<\lambda_1$, then for $u\in W^{1,p}_{0}(0,\pi_{p})$, the Sobolev-type inequality holds:
\begin{align}\label{Sobolev}
C(p,\lambda)\left(\max_{x \in [0,\pi_{p}]}\vert u(x)\vert\right)^{p}\leq \int_{0}^{\pi_{p}}\vert u'\vert^{p}\,dx-\lambda\int_{0}^{\pi_{p}}\vert u\vert^{p}\,dx.
\end{align}
Moreover, the constant $C(p,\lambda)$ is sharp.
\end{corollary}
\begin{rem}
Poincar\'{e} inequality yields
\begin{align}
\label{upper}
 \int_{0}^{\pi_{p}}\vert u'\vert^{p}\,dx-\lambda\int_{0}^{\pi_{p}}\vert u\vert^{p}\,dx
\leq
\begin{cases}
 \int_{0}^{\pi_{p}}\left|u'\right|^{p}\,dx,& 0< \lambda<\lambda_1,\\
\left(1-\frac{\lambda}{\lambda_1}\right)\int_{0}^{\pi_{p}}\left|u'\right|^{p}\,dx,& \lambda\leq 0.
\end{cases}
\end{align}
Hence, the inequality \eqref{Sobolev} describes the embedding $W^{1,p}_{0}(0,\pi_{p})\subset L^{\infty}(0,\pi_{p})$.
\end{rem}
\section{Proof of Theorem \ref{thm1}}
To prove Theorem \ref{thm1}, we consider the minimization problem:
\begin{align}\label{min1}
\inf_{u\in W^{1,p}_{0}(0,\pi_{p}),\, u\neq 0}\frac{J(u)}
{\Vert u\Vert^{p}_{L^{\infty}(0,\pi_{p})}}=
\inf_{u\in W^{1,p}_{0}(0,\pi_{p}),\, \Vert u\Vert_{L^{\infty}(0,\pi_{p})}=1}J(u),
\end{align}
where 
\begin{align*}
 J(u)=\int_{0}^{\pi_{p}}\vert u'\vert^{p}\,dx-\lambda\int_{0}^{\pi_{p}}\vert u\vert^{p}\,dx.
\end{align*}
Including \eqref{upper}, the following inequality holds for any $u \in W^{1,p}_0(0,\pi_p)$:
\begin{equation}
\label{bounded}
\left(1-\frac{\lambda_+}{\lambda_1}\right)\int_0^{\pi_p}|u'|^p\,dx
\leq J(u) \leq
\left(1-\frac{\lambda_-}{\lambda_1}\right)\int_0^{\pi_p}|u'|^p\,dx,
\end{equation}
where $\lambda_+=\max(\lambda,0)$ and $\lambda_-=\min(\lambda,0)$
for $\lambda<\lambda_1$.
In particular, $J(u)>0$ for $u \in W^{1,p}_0(0,\pi_p) \setminus \{0\}$
and $J(u)=0$ if and only if $u=0$.
\begin{lem}\label{lem1}
Assume $1<p<\infty$ and $\lambda<\lambda_1$. Then, the infimum of \eqref{min1} is attained.
\end{lem}
\begin{proof}
We define 
\begin{align*}
W=&
\left\{
u\in W^{1,p}_{0}(0,\pi_{p})\,|\,\|u\|_{L^{\infty}(0,\pi_{p})}=1\right\},\\
W_{R}=&\left\{
u\in W^{1,p}_{0}(0,\pi_{p})\,|\,\|u\|_{L^{\infty}(0,\pi_{p})}=1,\,\,
\int_{0}^{\pi_{p}}\left|u'\right|^{p}\,dx\leq R
\right\}, \quad R>0.
\end{align*}
Let $u_{0}\in W$. We put $R$ so large that
\begin{align}\label{R}
\left(1-\frac{\lambda_+}{\lambda_1}\right)R>\left(1-\frac{\lambda_-}{\lambda_1}\right)\int_{0}^{\pi_{p}}\left|u'_{0}\right|^{p}\,dx.
\end{align}
Take $u\in W\setminus W_{R}$. Then from \eqref{bounded} and \eqref{R}, we have 
\begin{align*}
J(u)\geq \left(1-\frac{\lambda_+}{\lambda_1}\right)\int_{0}^{\pi_{p}}\left|u'\right|^{p}\,dx
> \left(1-\frac{\lambda_+}{\lambda_1}\right)R> 
\left(1-\frac{\lambda_-}{\lambda_1}\right)\int_{0}^{\pi_{p}}\left|u'_{0}\right|^{p}\,dx \geq J(u_{0}).
\end{align*}
Hence, we obtain
\begin{align}\label{BR}
\inf_{u \in W\setminus W_{R}}J(u) > J(u_{0})\geq \inf_{u\in W}J(u).
\end{align}
From \eqref{BR} we have
\begin{align*}
\inf_{u\in W}J(u)=\inf_{u\in W_{R}}J(u).
\end{align*}

We know the set
\begin{align*}
\left\{
u\in W^{1,p}_{0}(0,\pi_{p})\,|\,
\int_{0}^{\pi_{p}}\left|u'\right|^{p}\,dx\leq R
\right\}
\end{align*}
is a weakly compact set. While from the Sobolev embedding, $W^{1,p}(0,\pi_{p})$ is compactly embedded in $L^{\infty}(0,\pi_{p})$, hence the subset $W_{R}$ is a weakly closed set. Thus, $W_{R}$ is a weakly compact set. 
Since $J$ is a weakly lower semi-continuous functional on $W^{1,p}(0,\pi_{p})$ (because $\Vert u'\Vert^{p}_{L^{p}(0,\pi_{p})}$ is weakly lower semi-continuous and $\Vert u\Vert^{p}_{L^{p}(0,\pi_{p})}$ is weakly continuous), the infimum of $J$ is attained on $W_{R}$. 
\end{proof}
We will denote by $\tilde{C}(p,\lambda)$ the infimum of \eqref{min1}. 
Note that $\tilde{C}(p,\lambda)>0$ by \eqref{bounded}. 
\begin{lem}\label{lem2}
Suppose $1<p<\infty$, $\lambda<\lambda_1$, and \eqref{SLP} has a nontrivial solution. 
Then, it holds
\begin{align}\label{Lya}
\|r_{+}\|_{L^{1}(0,\pi_{p})}> \tilde{C}(p,\lambda).
\end{align}
\end{lem}
\begin{proof}
Let $u$ be a nontrivial solution of \eqref{SLP}. Multiplying both sides of \eqref{SLP} by $u$ and integrating it over $(0,\pi_{p})$, we obtain
\begin{align*}
J(u)=\int_{0}^{\pi_p}r(x)\left|u(x)\right|^{p}\,dx;
\end{align*}
hence from \eqref{bounded},
\begin{align}
\label{juupper}
0<J(u) \leq \int_{0}^{\pi_p}r_+(x)\left|u(x)\right|^{p}\,dx
<\Vert u\Vert^{p}_{L^{\infty}(0,\pi_{p})}\int_{0}^{\pi_p}r_{+}(x)\,dx.
\end{align}


Here, the last strict inequality is proved as follows.
Assume the equality holds.
Then, $r_+(x)(|u(x)|^p-\|u\|_{L^{\infty}(0,\pi_{p})}^p)=0$ in $(0,\pi_p)$; hence 
$r_+(x)=0$ or $u(x)=M$ for each $x \in (0,\pi_p)$, 
where $M=\pm \|u\|_{L^{\infty}(0,\pi_{p})} \neq 0$. 
Set $P=\{x \in (0,\pi_p)\ |\ r_+(x)>0\}$.
By \eqref{juupper}, we may suppose $P \neq \emptyset$.
Moreover, $P \neq (0,\pi_p)$, since if not so,  
then $u\equiv M \neq 0$ on $(0,\pi_p)$ and $u(0)=u(\pi_p)=0$, which is impossible.
Now, by the equation \eqref{SLP}, $r_+(x)=r(x)=-\lambda$ (and $\lambda<0$) on $P$,
since $u \equiv M$ on the open set $P$.
On the other hand, $r_+(x)=0$ on $(0,\pi_p)\setminus P$. 
Hence, $r_+$ is surjective from $(0,\pi_p)$ to $\{0,-\lambda\}$, where $\lambda \neq 0$.
This contradicts the continuity of $r_+$. 


It follows from \eqref{juupper} and the definition of $\tilde{C}(p,\lambda)$ that
\begin{align}\label{Lineq}
J(u)
<  \tilde{C}(p,\lambda)^{-1}J(u)\int_{0}^{\pi_p}r_{+}(x)\,dx.
\end{align}
Since $u \neq 0$, dividing both sides of \eqref{Lineq} by $J(u)>0$,
we obtain the assertion.
\end{proof}

\begin{rem}
Lemma \ref{lem2} assures that 
there exists no nontrivial solution of \eqref{SLP} if 
\begin{align*}
\int_{0}^{\pi_{p}}r_{+}(x)dx \leq \tilde{C}(p,\lambda).
\end{align*}
\end{rem}

\begin{lem} \label{lem3}
The estimate \eqref{Lya} is sharp, in the sense that there exists 
$r\in C[0,\pi_{p}]$ for which \eqref{SLP} has a nontrivial solution such that the left-hand side of \eqref{Lya} can be as close as possible to the right-hand side of \eqref{Lya}.  
\end{lem}
\begin{proof}
Let $u_{*}\in W^{1,p}(0,\pi_{p})$ be 
the minimizer of \eqref{min1}, whose existence is assured by Lemma \ref{lem1}.
Moreover, let $c\in (0,\pi_{p})$ be a point with 
$|u_*(c)|=\max_{x \in [0,\pi_p]}{|u_*(x)|}$. 
We define the function $r_{\delta}$ as
\begin{align*}
r_{\delta}(x)=
\begin{cases}
0, & 0\leq x<c-\delta,\,\,c+\delta<x\leq \pi_{p},\\
\frac{1}{\delta^2}(x-c+\delta), &c-\delta\leq x< c,\\
-\frac{1}{\delta^2}(x-c-\delta), &c\leq x\leq c+\delta,
\end{cases}
\end{align*} 
where $\delta>0$ satisfies $0<c-\delta$ and $c+\delta<\pi_{p}$.
Note that it holds
\begin{align}\label{intrd}
\int_{0}^{\pi_{p}}r_{\delta}(x)\,dx=\int_{c-\delta}^{c+\delta}r_{\delta}(x)\,dx=1.
\end{align}

As in the proof of Lemma \ref{lem1}, we can show the existence of 
minimizer in $W^{1,p}_0(0,\pi_p)$ of
\begin{align*}
\inf_{\phi\in W^{1,p}_{0}(0,\pi_{p}),\phi\neq 0}\frac{\int_{0}^{\pi_{p}}\left|\phi'(x)\right|^{p}\,dx-\lambda\int_{0}^{\pi_{p}}\left|\phi(x)\right|^{p}\,dx}
{\int_{0}^{\pi_{p}}r_{\delta}(x)\left|\phi(x)\right|^{p}\,dx} =:\alpha(\delta).
\end{align*}
Thus, we emphasize that 
\eqref{SLP} has a nontrivial solution (as assertion) for
\begin{align*}
r(x)=\alpha(\delta)r_{\delta}(x) \in C[0,\pi_p].
\end{align*}

From \eqref{intrd}, for arbitrarily small $\epsilon>0$ there exists $\delta=\delta(\epsilon)>0$ such that the following inequality holds:  
\begin{align*}
\frac{\int_{0}^{\pi_{p}}\left|u_{*}'(x)\right|^{p}\,dx-\lambda\int_{0}^{\pi_{p}}\left|u_{*}(x)\right|^{p}\,dx}{\int_{0}^{\pi_{p}}r_{\delta}(x)\left|u_{*}(x)\right|^{p}\,dx}-\epsilon
<\frac{\int_{0}^{\pi_{p}}\left|u_{*}'(x)\right|^{p}\,dx-\lambda\int_{0}^{\pi_{p}}\left|u_{*}(x)\right|^{p}\,dx}{\|u_{*}\|_{L^{\infty}(0,\pi_{p})}^{p}}.
\end{align*}
Hence from the definition of $\alpha(\delta)$ and $u_{*}$, we obtain
\begin{align*}
\alpha(\delta(\epsilon))-\epsilon
<\tilde{C}(p,\lambda).
\end{align*}
From above inequality, we have
\begin{align*}
\int_{0}^{\pi_{p}}r(x)\,dx=\alpha(\delta(\epsilon))\int_{0}^{\pi_{p}}r_{\delta(\epsilon)}(x)\,dx=
\alpha(\delta(\epsilon))< \tilde{C}(p,\lambda)+\epsilon.
\end{align*}
Hence, by $r=r_+$ and \eqref{Lya}, we obtain
\begin{align*}
\tilde{C}(p,\lambda)<\int_{0}^{\pi_{p}}r(x)\,dx < \tilde{C}(p,\lambda)+\epsilon.
\end{align*}
Since $\epsilon>0$ is arbitrarily small, this is the desired conclusion. 
\end{proof}
From Lemmas \ref{lem1}--\ref{lem3} above, 
to complete the proof of Theorem \ref{thm1},
all we have to do is to show 
\begin{equation}
\label{eq:C=C}
\tilde{C}(p,\lambda)=C(p,\lambda).
\end{equation}
Corollary \ref{cor1} immediately follows from the definition of $\tilde{C}(p,\lambda)$.

We define the following sets for $y \in (0,\pi_p)$:
\begin{align}\label{MWy}
W(y)=&\left\{
u\in W^{1,p}_{0}(0,\pi_{p})\,|\, u(y)=1
\right\},\\
M(y)=&\left\{
u\in W^{1,p}_{0}(0,\pi_{p})\,|\, \max_{x\in [0,\pi_{p}]}\left|u(x)\right|=u(y)=1
\right\}.\nonumber
\end{align}
Then, $M(y) \subset W(y)$ and
\begin{align}\label{infM}
\tilde{C}(p,\lambda)=
\inf_{y\in (0,\pi_{p})}\inf_{u\in M(y)}J(u)=\inf_{y\in (0,\pi_{p}/2]}\inf_{u\in M(y)}J(u),
\end{align}
since if $u(\cdot)\in W^{1,p}_{0}(0,\pi_{p})$ attains the infimum of \eqref{infM}, 
then so does $u(\pi_{p}/2-\cdot)\in W^{1,p}_{0}(0,\pi_{p})$. 

Instead of handling \eqref{infM} directly, we consider the relaxed problem:
\begin{align}
\label{relaxed}
\inf_{y\in(0,\pi_{p}/2]}\inf_{u\in W(y)}J(u)=\inf_{y\in(0,\pi_{p}/2]}F(y),
\end{align}
where $F(y)=\inf_{u\in W(y)}J(u)$.

To show \eqref{eq:C=C}, we will divide the proof into three cases
$0<\lambda<\lambda_1$, $\lambda=0$, and $\lambda<0$.

\subsection{The case $0<\lambda<\lambda_1$}
In this case, we show
\begin{proposition}
\label{propC}
\begin{align}\label{bestconst}
\tilde{C}(p,\lambda)=C(p,\lambda)=
2K^{p-1}\left(\cot_{p}\left(\frac{K\pi_{p}}{2}\right)\right)^{p-1},
\end{align}
where $K=(\lambda/\lambda_1)^{1/p} \in (0,1)$.
\end{proposition}

\begin{lem}\label{lemuy}
For $y\in (0,\pi_{p}/2]$, 
$\inf_{u\in W(y)}J(u)$ is attained by the following $u_y\in W(y)$. 
\begin{align}\label{uy}
u_{y}(x)=
\begin{cases}
\frac{\sin_{p}Kx}{\sin_{p}Ky},& 0\leq x< y,\\
\frac{\sin_{p}K(\pi_{p}-x)}{\sin_{p}K(\pi_{p}-y)},&y\leq x\leq \pi_{p}.
\end{cases}
\end{align}
\end{lem}
\begin{proof}
The existence of the minimizer $u_y\in  W(y)$ of $J$ is assured in a similar way to Lemma \ref{lem1}. 

Let $\varphi\in W^{1,p}_{0}(0,\pi_{p})$ be an arbitrary element satisfying $\varphi(0)=\varphi(\pi_{p})=\varphi(y)=0$. Then the first variation of $J(u)$ at $u\in W(y)$ is
\begin{align*}
J'(u)[\varphi]=&p\int_{0}^{\pi_{p}}\left(\left|u'(x)\right|^{p-2}u'(x)\varphi'(x)-\lambda\left|u(x)\right|^{p-2}u(x)\varphi(x)\right)\,dx\\
=&p\left[\left|u'\right|^{p-2}u'\varphi\right]_{0}^{y}-p\int_{0}^{y}\left(\left(
\left|u'\right|^{p-2}u'
\right)'+\lambda\left|u\right|^{p-2}u\right)\varphi\, dx\\
&\qquad +p\left[\left|u'\right|^{p-2}u'\varphi\right]_{y}^{\pi_{p}}-p\int_{y}^{\pi_{p}}\left(\left(
\left|u'\right|^{p-2}u'
\right)'+\lambda\left|u\right|^{p-2}u\right)\varphi\, dx\\
=&-p\int_{0}^{y}\left(\left(
\left|u'\right|^{p-2}u'
\right)'+\lambda\left|u\right|^{p-2}u\right)\varphi\, dx
-p\int_{y}^{\pi_{p}}\left(\left(
\left|u'\right|^{p-2}u'
\right)'+\lambda\left|u\right|^{p-2}u\right)\varphi\, dx.
\end{align*}
Since $J(u_{y})=\inf_{u\in W(y)}J(u)$, $u_{y}$ satisfies $J'(u_y)[\varphi]=0$, and hence
\begin{align}\label{left}
\begin{cases}
\left(
\left|u_y'(x)\right|^{p-2}u_y'(x)
\right)'+\lambda\left|u_y(x)\right|^{p-2}u_y(x)=0,\quad x\in (0,y),\\
u_y(0)=0,\,u_y(y)=1,\\
\end{cases}
\end{align}
and
\begin{align}\label{right}
\begin{cases}
\left(
\left|u_y'(x)\right|^{p-2}u_y'(x)
\right)'+\lambda\left|u_y(x)\right|^{p-2}u_y(x)=0,\quad x\in (y,\pi_{p}),\\
u_y(y)=1,\,u_y(\pi_{p})=0.\\
\end{cases}
\end{align}
If we drop the condition $u_y(y)=1$ from \eqref{left} and \eqref{right}, we have 
from \eqref{ssh}, 
$u_y(x)=C_{1}\sin_{p}Kx$ for \eqref{left} and $u_y(x)=C_{2}\sin_{p}K(\pi_{p}-x)$ for \eqref{right}, where $C_{1}$ and $C_{2}$ are constants. 
From the condition $u_y(y)=1$, the constants $C_{1}$ and $C_{2}$ are uniquely determined as $C_{1}=(\sin_{p}Ky)^{-1}$ and $(\sin_{p}K(\pi_{p}-y))^{-1}$ respectively.  
\end{proof}
To establish the expression of $F(y)$ in terms of $y$, we prepare integral formulas
of the generalized trigonometric functions.
\begin{lem}\label{lemI1I2}
For any $z \in {\Bbb R}$,
\begin{align*}
 \int_{0}^{z}\left|\cos_{p}t\right|^{p}\,dt=&\frac{1}{p}\left((p-1)z+\left|\cos_{p}z\right|^{p-2}\cos_{p}z\sin_{p}z\right),\\
\int_{0}^{z}\left|\sin_{p}t\right|^{p}\,dt=&\frac{1}{p}\left(
z-\left|\cos_{p}z\right|^{p-2}\cos_{p}z\sin_{p}z\right).
\end{align*}
\end{lem} 
\begin{proof}
The proof is same as that of Lemma 2 in \cite{Takeuchi2019}. Put 
\begin{align*}
I_{1}(z)=\int_{0}^{z}\left|\cos_{p}t\right|^{p}\,dt,\,\,\,I_{2}(z)=\int_{0}^{z}\left|\sin_{p}t\right|^{p}\,dt.
\end{align*}
Then we have from \eqref{plusone}, 
\begin{align}\label{I1I2-1}
I_{1}(z)+I_{2}(z)=z.
\end{align}
From \eqref{diff2} and \eqref{ssh}, we see 
\begin{align}\label{diff3}
\left(\left|\cos_{p}x\right|^{p-2}\cos_{p}x\right)'=-(p-1)\left|\sin_{p}x\right|^{p-2}\sin_{p}x;
\end{align}
hence
\begin{align}\label{I1I2-2}
 I_{2}(z)=&\int_{0}^{z}\left|\sin_{p}t\right|^{p-2}\left(\sin_{p}t\right)^2\,dt=
\int_{0}^{z}\frac{\left(\left|\cos_{p}t\right|^{p-2}\cos_{p}t\right)'}{-(p-1)}
\sin_{p}t\,dt\\
=&\left[
-\frac{\left|\cos_{p}t\right|^{p-2}\cos_{p}t}{p-1}\sin_{p}t
\right]_{0}^{z}+\int_{0}^{z}\frac{\left|\cos_{p}t\right|^{p-2}\cos_{p}t}{p-1}\cos_{p}t\,dt\nonumber\\
=&-\frac{\left|\cos_{p}z\right|^{p-2}\cos_{p}z}{p-1}\sin_{p}z+
\frac{I_{1}(z)}{p-1}.\nonumber
\end{align}
From \eqref{I1I2-1} and \eqref{I1I2-2}, we obtain the formulas.
\end{proof}
\begin{lem}
\label{lemF}
For $y \in (0,\pi_p/2]$, the function $F$ is expressed as follows:
\begin{align}\label{functionF}
F(y)=K^{p-1}\left(\left(\cot_{p}Ky\right)^{p-1}+\left|\cot_{p}K(\pi_{p}-y)\right|^{p-2}\cot_{p}K(\pi_{p}-y)\right).
\end{align}
\end{lem}
\begin{proof}
Since $F(y)=J(u_y)$, we have
\begin{align*}
F(y)=\int^{y}_{0}\left(\left|u'_{y}(x)\right|^{p}-\lambda \left|u_{y}(x)\right|^{p}\right)\,dx+
\int^{\pi_p}_{y}\left(\left|u'_{y}(x)\right|^{p}-\lambda \left|u_{y}(x)\right|^{p}\right)\,dx,
\end{align*}
where $u_{y}$ is as \eqref{uy} in Lemma \ref{lemuy}.
By Lemma \ref{lemI1I2}, the first term in the right-hand side is expressed as
\begin{align*}
\int^{y}_{0}\left(\left|u'_{y}(x)\right|^{p}\right.
&\left.-\lambda \left|u_{y}(x)\right|^{p}\right)\,dx\\
&=\frac{1}{\left(\sin_{p}Ky\right)^{p}}\int_{0}^{y}\left(K\cos_{p}Kx\right)^{p}\,dx-\frac{\lambda}{\left(\sin_{p}Ky\right)^{p}}
\int_{0}^{y}\left(\sin_{p}Kx\right)^{p}\,dx\\
&=\frac{K^{p-1}}{\left(\sin_{p}Ky\right)^{p}}\int_{0}^{Ky}\left(\cos_{p}t\right)^{p}\,dt
-\frac{\lambda K^{-1}}{\left(\sin_{p}Ky\right)^{p}}\int_{0}^{Ky}\left(\sin_{p}t\right)^{p}\,dt\\
&=K^{p-1}\left(\cot_{p}Ky\right)^{p-1}.
\end{align*}
Similarly, we obtain
\begin{align*}
\int^{\pi_p}_{y}\left(\left|u'_{y}(x)\right|^{p}-\lambda \left|u_{y}(x)\right|^{p}\right)\,dx
=K^{p-1}\left|\cot_{p}K(\pi_{p}-y)\right|^{p-2}\cot_{p}K(\pi_{p}-y).
\end{align*}
This completes the proof.
\end{proof}
We will evaluate the minimum of $F(y)$ for $y$.
\begin{lem}\label{monotone}
Put
\begin{align*}
f(x)=K^{p-1}\left|\cot_{p}{Kx}\right|^{p-2}\cot_{p}{Kx}, \quad x \in \left(0,\frac{\pi_p}{K}\right).
\end{align*}
Then,
\begin{align*}
\begin{cases}
\text{(i) $f(x)=-f(\frac{\pi_{p}}{K}-x)$, $x\in (0,\frac{\pi_{p}}{2K}]$,}\\
\text{(ii) $f(x)$ is monotone decreasing on $(0,\frac{\pi_{p}}{K})$,}\\
\text{(iii) $f(x)$ is convex on $(0,\frac{\pi_{p}}{2K}]$ and concave on $[\frac{\pi_{p}}{2K},\frac{\pi_{p}}{K})$.} 
\end{cases}
\end{align*}
\end{lem} 
\begin{proof}
The definition of $\cot_p{x}$ immediately yields (i).   
From \eqref{diff2} and \eqref{diff3} we obtain
\begin{align}\label{cot}
f'(x)=-\frac{(p-1)K^p}{\left(\sin_{p}{Kx}\right)^{p}}<0,\quad x\in \left(0,\frac{\pi_{p}}{K}\right).
\end{align}
Assertions (ii) and (iii) are obtained from \eqref{cot}. 
\end{proof}
\begin{lem}\label{lemmin}
The infimum of $F(y)$, $y \in (0,\pi_p/2]$, is attained at $y=\pi_{p}/2$.
Moreover,
$$F\left(\frac{\pi_p}{2}\right)=2K^{p-1}\left(\cot_{p}\left(\frac{K\pi_{p}}{2}\right)\right)^{p-1}
=C(p,\lambda).$$
\end{lem}

\begin{proof}
Let $f$ be the function defined in Lemma \ref{monotone}.
Differentiating \eqref{functionF}, i.e., $F(y)=f(y)+f(\pi_p-y)$,
in Lemma \ref{lemF}, we have $F'(y)=f'(y)-f'(\pi_p-y)$.

Let $y \in (0,\pi_p/2)$. 
In the case $\pi_p-y \leq \pi_p/(2K)$,
since $f'$ is increasing in $(0,\pi_p/(2K))$, it follows that 
$F'(y)=f'(y)-f'(\pi_p-y)<0$. Suppose $\pi_p-y > \pi_p/(2K)$.
The symmetry of $f'$ to $x=\pi_p/(2K)$ yields that $f'(\pi_p-y)=f'(z)$,
where $z=(1/K-1)\pi_p+y \in (y,\pi_p/(2K))$. 
Since $f'$ is increasing in $(0,\pi_p/(2K))$, 
it follows that $F'(y)=f'(y)-f'(z)<0$.
From above, we conclude that $F'(y)<0$ in $(0,\pi_p/2)$
and $F(\pi_p/2)=2f(\pi_p/2)$ is the minimum of $F(y)$.
\end{proof}

We now turn to the relaxed problem \eqref{relaxed}.
From Lemma \ref{lemmin}, we obtain
\begin{align*}
\inf_{y\in(0,\pi_{p}/2]}\inf_{u\in W(y)}J(u)
&=\inf_{y\in(0,\pi_{p}/2]}F(y)=F\left(\frac{\pi_p}{2}\right)=C(p,\lambda).
\end{align*}
Since $M(y)\subset W(y)$, in general, it holds
\begin{align*}
\inf_{u\in W(y)}J(u)\leq \inf_{u\in M(y)}J(u);
\end{align*}
hence $C(p,\lambda) \leq \tilde{C}(p,\lambda)$.
However, in the case $y=\pi_{p}/2$, the maximum of $u_{y}$ is attained at $x=\pi_{p}/2$ and its value is 1 (so, $u_{\pi_{p}/2} \in M(\pi_{p}/2)$). Therefore,
\begin{align*}
C(p,\lambda)=\inf_{u\in W(\pi_{p}/2)}J(u)
=\inf_{u\in M(\pi_{p}/2)}J(u) \geq \tilde{C}(p,\lambda).
\end{align*}
Thus, $\tilde{C}(p,\lambda)=C(p,\lambda)$, and we have shown Proposition \ref{propC}.

\begin{rem}\label{rem3}
The function $u_y$ does not always belong to $M(y)$.
Indeed, assume $K \in (1/2,1)$, i.e., $\lambda \in (\lambda_1/2^p,\lambda_1)$. 
Let $y \in (0,\pi_p/2)$ be a number satisfying
\begin{align*}
0<y<\left(1-\frac{1}{2K}\right)\pi_{p}<\frac{\pi_{p}}{2}.
\end{align*}
Then, since $K(\pi_p-y)<\pi_p/2$,  
\begin{align*}
u_{y}\left(\left(1-\frac{1}{2K}\right)\pi_{p}\right)=\frac{\sin_{p}(\pi_p/2)}{\sin_{p}K(\pi_{p}-y)}>\frac{1}{\sin_{p}(\pi_p/2)}=1;
\end{align*}
hence $u_{y}\not\in M(y)$. On the other hand, if $y=\pi_{p}/2$, we have $u_{\pi_{p}/2}\in M(\pi_{p}/2)$. 
Figure \ref{fig4} shows the graph of $u_{y}$ which does not belong to $M(y)$.
\begin{figure}[htb]
 \begin{center}
 \includegraphics[scale=0.7 ]{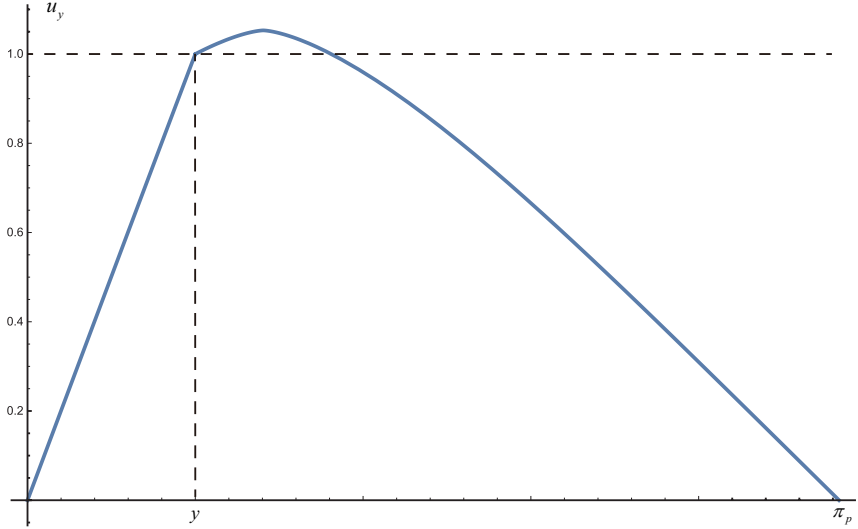}
\caption{The graph of $u_{y}$ which does not belong to $M(y)$.}
\label{fig4}
 \end{center}
\end{figure}
\end{rem}
\subsection{The case $\lambda=0$}
In this case, we show
\begin{proposition}\label{prop=}
\begin{align*}
\tilde{C}(p,\lambda)=C(p,\lambda)=\frac{2^p}{\pi_p^{p-1}}.
\end{align*}
\end{proposition}

As with Lemma \ref{lemuy}, we obtain the next lemma.
\begin{lem}\label{lemuy2=}
For $y \in (0,\pi_p/2]$, $\inf_{u\in W(y)}J(u)$ is attained by the following $u_y\in W(y)$. 
\begin{align}\label{ush=}
u_{y}(x)=
\begin{cases}
\frac{x}{y},& 0\leq x< y,\\
\frac{\pi_{p}-x}{\pi_{p}-y},& y\leq x\leq \pi_{p}. 
\end{cases}
\end{align}
\end{lem}
\begin{proof}
The existence of the minimizer $u_y\in W(y)$ of $J$ is assured in a similar way to Lemma \ref{lem1}. 
Since $u_{y}$ is a minimizer of $J$ for all $u\in W(y)$, as in the proof of Lemma \ref{lemuy}, $u_{y}$ satisfies \eqref{left} and \eqref{right}. In this case, solutions of \eqref{left} and \eqref{right} are uniquely expressed as 
the functions of \eqref{ush=}, respectively.
\end{proof}
\begin{lem}
\label{hF=}
For $y \in (0,\pi_p/2]$, the function $F$ is expressed as follows:
\begin{align*}
F(y)=\frac{1}{y^{p-1}}+\frac{1}{(\pi_p-y)^{p-1}}.
\end{align*}
\end{lem}
\begin{proof}
Since $F(y)=J(u_y)$, where $u_{y}$ is as \eqref{ush=} in Lemma \ref{lemuy2=},
we have
\begin{align*}
F(y)&=\int^{y}_{0}\left|u'_{y}(x)\right|^{p}\,dx+
\int^{\pi_p}_{y}\left|u'_{y}(x)\right|^{p}\,dx
=\frac{1}{y^{p-1}}+\frac{1}{(\pi_p-y)^{p-1}}.
\end{align*}
This completes the proof.
\end{proof}
\begin{lem}
The infimum of $F(y)$, $y \in (0,\pi_p/2]$, is attained at $y=\pi_{p}/2$.
Moreover,
$$F(\pi_p/2)=\frac{2^p}{\pi_p^{p-1}}
=C(p,\lambda).$$
\end{lem}
\begin{proof}
By Lemma \ref{hF=} and Jensen's inequality with $1/x^{p-1}$, we have
\begin{align*}
F(y) \geq \frac{2^p}{(y+\pi_p-y)^{p-1}}=\frac{2^p}{\pi_p^{p-1}}.
\end{align*}
The equality holds when $y=\pi_p-y$, i.e., $y=\pi_p/2$.
\end{proof}

Similar argument to Proposition \ref{propC} yields Proposition \ref{prop=}.

\subsection{The case $\lambda<0$}
In this case, we show
\begin{proposition}\label{propD}
\begin{align*}
\tilde{C}(p,\lambda)=C(p,\lambda)=
2K^{p-1}\left(\coth_{p}\left(\frac{K\pi_{p}}{2}\right)\right)^{p-1},
\end{align*}
where $K=(-\lambda/\lambda_1)^{1/p} \in (0,\infty)$.
\end{proposition}
As with Lemma \ref{lemuy}, we obtain the next lemma.
\begin{lem}\label{lemuy2}
For $y \in (0,\pi_p/2]$, $\inf_{u\in W(y)}J(u)$ is attained by the following $u_y\in W(y)$. 
\begin{align}\label{ush}
u_{y}(x)=
\begin{cases}
\frac{\sinh_{p}Kx}{\sinh_{p}Ky},& 0\leq x< y,\\
\frac{\sinh_{p}K\left(\pi_{p}-x\right)}{\sinh_{p}K\left(\pi_{p}-y\right)},& y\leq x\leq \pi_{p}. 
\end{cases}
\end{align}
\end{lem}
\begin{proof}
The proof is same as that of Lemma \ref{lemuy2=}.
\end{proof}
\begin{lem}\label{leminteg2}
For any $z \in {\Bbb R}$,
\begin{align*}
\int_{0}^{z}\left(\cosh_{p}t\right)^{p}\,dt&=\frac{1}{p}\left(
(p-1)z+(\cosh_{p}z)^{p-1}\sinh_{p}z
\right),\\
\int_{0}^{z}\left|\sinh_{p}t\right|^{p}\,dt&=\frac{1}{p}\left(
-z+(\cosh_{p}z)^{p-1}\sinh_{p}z
\right).
\end{align*}
\end{lem}
\begin{proof}
Put 
\begin{align*}
I_{1}(z)=\int_{0}^{z}\left(\cosh_{p}t\right)^{p}\,dt,\quad I_{2}(z)=\int_{0}^{z}\left|\sinh_{p}t\right|^{p}\,dt.
\end{align*}
Then we have from \eqref{minusone}, 
\begin{align}\label{I1-I2}
I_{1}(z)-I_{2}(z)=z.
\end{align}
From \eqref{diff2h} and \eqref{ssh2}, we have 
\begin{align*}
\left(\left(\cosh_{p}x\right)^{p-1}\right)'=(p-1)\left|\sinh_{p}x\right|^{p-2}\sinh_{p}x,
\end{align*}
hence 
\begin{align}\label{I2h}
I_{2}(z)&=\int_{0}^{z}\left|\sinh_{p}t\right|^{p-2}(\sinh_{p}t)^2\,dt=\int_{0}^{z}\frac{\left(\left(\cosh_{p}t\right)^{p-1}\right)'\sinh_{p}t}{p-1}\,dt\\
&=\left[\frac{\left(\cosh_{p}t\right)^{p-1}}{p-1}\sinh_{p}t\right]_{0}^{z}-\int_{0}^{z}\frac{\left(\cosh_{p}t\right)^{p-1}}{p-1}\cosh_{p}t\,dt\nonumber\\
&=\frac{\left(\cosh_{p}z\right)^{p-1}\sinh_{p}z}{p-1}-\frac{I_{1}(z)}{p-1}.\nonumber
\end{align}
From \eqref{I1-I2} and \eqref{I2h}, we obtain the formulas.
\end{proof}
\begin{lem}
\label{hF}
For $y \in (0,\pi_p/2]$, the function $F$ is expressed as follows:
\begin{align*}
F(y)=K^{p-1}\left(\left(\coth_{p}Ky\right)^{p-1}+\left(\coth_{p}K(\pi_{p}-y)
\right)^{p-1}\right).
\end{align*}
\end{lem}
\begin{proof}
Since $F(y)=J(u_y)$, we have
\begin{align*}
F(y)=\int^{y}_{0}\left(\left|u'_{y}(x)\right|^{p}-\lambda \left|u_{y}(x)\right|^{p}\right)\,dx+
\int^{\pi_p}_{y}\left(\left|u'_{y}(x)\right|^{p}-\lambda \left|u_{y}(x)\right|^{p}\right)\,dx,
\end{align*}
where $u_{y}$ is as \eqref{ush} in Lemma \ref{lemuy2}.
By Lemma \ref{leminteg2}, the first term in the right-hand side is expressed as
\begin{align*}
\int^{y}_{0}\left(\left|u'_{y}(x)\right|^{p}\right.
&\left.-\lambda \left|u_{y}(x)\right|^{p}\right)\,dx\\
&=\frac{1}{\left(\sinh_{p}Ky\right)^{p}}\int_{0}^{y}\left(K\cosh_{p}Kx\right)^{p}\,dx-\frac{\lambda}{\left(\sinh_{p}Ky\right)^{p}}
\int_{0}^{y}\left(\sinh_{p}Kx\right)^{p}\,dx\\
&=\frac{K^{p-1}}{\left(\sinh_{p}Ky\right)^{p}}\int_{0}^{Ky}\left(\cosh_{p}t\right)^{p}\,dt
-\frac{\lambda K^{-1}}{\left(\sinh_{p}Ky\right)^{p}}\int_{0}^{Ky}\left(\sinh_{p}t\right)^{p}\,dt\\
&=K^{p-1}\left(\coth_{p}Ky\right)^{p-1}.
\end{align*}
Similarly, we obtain
\begin{align*}
\int^{\pi_p}_{y}\left(\left|u'_{y}(x)\right|^{p}-\lambda \left|u_{y}(x)\right|^{p}\right)dx
=K^{p-1}\left(\coth_{p}K(\pi_{p}-y)\right)^{p-1}.
\end{align*}
This completes the proof.
\end{proof}
\begin{lem}
The infimum of $F(y)$, $y \in (0,\pi_p/2]$, is attained at $y=\pi_{p}/2$.
Moreover,
$$F(\pi_p/2)=2K^{p-1}\left(\coth_{p}\left(\frac{K\pi_{p}}{2}\right)\right)^{p-1}
=C(p,\lambda).$$
\end{lem}

\begin{proof}
Let $f(x)=K^{p-1}(\coth_p{Kx})^{p-1}$. Since
$$f'(x)=-\frac{(p-1)K^p}{(\sinh_p{Kx})^p}$$
is increasing in $(0,\infty)$, 
the function $f(x)$ is convex in $(0,\infty)$.
Thus, by Lemma \ref{hF} and Jensen's inequality, we have
\begin{align*}
F(y)=f(y)+f(\pi_p-y) \geq 2f\left(\frac{y+\pi_p-y}{2}\right)
=2 K^{p-1}\left(\coth_p{\frac{K\pi_p}{2}}\right)^{p-1}.
\end{align*}
The equality holds when $y=\pi_p-y$, i.e., $y=\pi_p/2$.
\end{proof}


Similar argument to Proposition \ref{propC} yields Proposition \ref{propD}.

\begin{rem}
Contrary to the case $0<\lambda<\lambda_1$ (see Remark \ref{rem3}), we can show $u_{y}\in M(y)$ for all $\lambda<0$.
Indeed, it is obvious that $(\sinh_{p}Kx)/(\sinh_{p}Ky)$ is monotone increasing on $(0,y]$ and $(\sinh_{p}K(\pi_{p}-x))/(\sinh_{p}(\pi_{p}-y))$ is monotone decreasing on $[y,\pi_{p}/2]$.
Therefore, $\max_{x\in (0,\pi_{p}/2]}u_{y}(x)=u_{y}(y)=1$.
\end{rem}


\end{document}